\documentclass[a4paper, 11pt]{amsart}
\usepackage[utf8]{inputenc}
\usepackage[T1]{fontenc}
\usepackage{lmodern}
\usepackage[english]{babel}
\usepackage[dvipsnames]{xcolor}
\usepackage{adjustbox}
\usepackage{letltxmacro}
\usepackage{todonotes}
\LetLtxMacro\todonotestodo\todo
\renewcommand{\todo}[2][]{\todonotestodo[#1]{TODO: {#2}}}
\usepackage{enumerate}
\usepackage{hyperref}
\usepackage{breakurl}
\usepackage{url}
\usepackage{mathrsfs}
\usepackage{amssymb}
\usepackage{ stmaryrd }

\theoremstyle{definition}
\newtheorem{theorem}{Theorem}

\newtheorem{lemma}{Lemma}[section]
\newtheorem{proposition}[lemma]{Proposition}
\newtheorem{corollary}[lemma]{Corollary}
\newtheorem{fact}[lemma]{Fact}

\newtheorem{remark}[lemma]{Remark}

\newtheorem*{claim*}{Claim}

\newtheorem*{theorem*}{Theorem}
\newtheorem*{corollary*}{Corollary}
\newtheorem*{lemma*}{Lemma}
\newtheorem*{remark*}{Remark}
\newtheorem*{question*}{Question}

\newtheorem{problem}[lemma]{Problem}

\DeclareMathOperator{\pic}{pic}

\DeclareMathOperator{\dep}{depth}
\DeclareMathOperator{\Spec}{Spec}
\DeclareMathOperator{\hei}{ht}
\DeclareMathOperator{\dw}{Div_{Weil}}
\DeclareMathOperator{\dc}{Div_{Cart}}
\DeclareMathOperator{\Pic}{Pic}
\DeclareMathOperator{\Clw}{Cl_{Weil}}
\DeclareMathOperator{\Clc}{Cl_{Cart}}
\DeclareMathOperator{\Cl}{Cl}

\title{On local divisor class groups of complete intersections}


\author{Daniel Windisch}
\thanks{During the preparation of this manuscript, the author was supported by the Austrian Science Fund (FWF), project P~5510-N26, which is part of the Special Research Program ``Quasi-Monte Carlo Methods: Theory and Applications''.}

\begin{document}

\maketitle

\begin{abstract}
Samuel conjectured in 1961 that a (Noetherian) local complete intersection ring that is a UFD in codimension at most three is itself a UFD. It is said that Grothendieck invented local cohomology to prove this fact. Following the philosophy that a UFD is nothing else than a Krull domain (that is, a normal domain, in the Noetherian case) with trivial divisor class group, we take a closer look at the Samuel--Grothendieck Theorem and prove the following generalization: Let $A$ be a local Cohen--Macaulay ring.
\begin{enumerate}
\item $A$ is a normal domain if and only if $A$ is a normal domain in codimension at most $1$.
\item Suppose that $A$ is a normal domain and a complete intersection. Then the divisor class group of $A$ is a subgroup of the projective limit of the divisor class groups of the localizations $A_p$, where $p$ runs through all prime ideals of height at most $3$ in $A$.
\end{enumerate}
We use this fact to describe for an integral Noetherian locally complete intersection scheme $X$ the gap between the groups of Weil and Cartier divisors, generalizing in this case the classical result that these two concepts coincide if $X$ is locally a UFD.
\end{abstract}

\section{Introduction}

It is a famous theorem by Auslander and Buchsbaum~\cite{regular-UFD} based on work by Nagata~\cite{Nagata} that every regular local ring is a unique factorization domain (UFD). This is probably the most prominent instance of a result linking a geometric property with a statement about factorization in local rings.

Lipman~\cite{Lipman} pointed out that, at least for projective normal closed subvarieties $V$ of a projective space (over any field), there is even a connection between the global geometric behaviour and the unique factorization property of local rings. In this setting, the local ring at the vertex of the projective cone over $V$ is a UFD if and only if every irreducible subvariety of  codimension one is cut out (scheme-theoretically) by a hypersurface of the ambient space.

However, the connection seems to lie deeper: The theory of divisors plays a crucial role in both, algebraic geometry and the combinatorial theory of non-unique factorizations in integral domains. Indeed, factorizations of non-units as products of irreducible elements in a normal Noetherian domain $A$ (or, more generally, in a Krull domain) only depend on the divisor class group of $A$ and the distribution of prime divisors over the classes, see~\cite[Chapters 3, 4 and 7]{GHK}. 

In view of this fact, the study of divisor class groups became prominent in geometric contexts, also apart from algebraic geometry: Kainrath~\cite{Kainrath} studied the distribution of prime divisor for a wide range of integral $k$-algebras. Remaining cases were settled by Fadinger-Held and the present author for affine toric varieties and their non-Noetherian analogues~\cite{KrullPrimeDiv, affine}, where they also give explicit descriptions of divisor class groups of one-dimensional toric varieties.

Combinatorial and arithmetic applications of these and similar results can be found in a growing number of papers. A celebrated work in this context is Kainrath's Theorem~\cite{KainrathFact} for Krull monoids with infinite divisor class group and prime divisors in all classes. We also want to mention the arithmetic investigations of~\cite[Section 5]{weaklyKrull2} which make direct use of the results mentioned in the preceding paragraph.

Not every Noetherian local ring that is a complete intersection is necessarily regular. In view of the above remarks, it is therefore reasonable to ask which local complete intersection rings are UFDs or, more generally, what can be said about their divisor class groups. Partial answers to this are due to Andreotti--Salmon~\cite{Andreotti-Salmon}, Lefschetz~\cite{Lefschetz} and Severi~\cite{Severi}. 
These results have eventually led to the following theorem -- conjectured by Samuel~\cite[p. 172]{Samuel} and proven by Grothendieck~\cite[p. 132]{SGA2} --, where we say that a property is satisfied by a ring $A$ \textit{in codimension} $n$ if all localizations of $A$ at prime ideals of height $n$ have the property.

\begin{theorem*}[Samuel--Grothendieck]
Let $A$ be a local complete intersection ring such that $A$ is a UFD in codimension at most $3$. Then $A$ is a UFD.
\end{theorem*}

In a 1968 survey on unique factorizations~\cite{SamuelSurvey}, Samuel notes that Grothendieck uses the most ``powerful methods of his theory of schemes'' and that no purely ring-theoretic proof was known at the time. It was only in 1994 that F.W. Call~\cite{Call} was able to bypass the scheme-theoretic machinery and prove Samuel's conjecture by purely algebraic means.

A key ingredient in both, Grothendieck's and Call's proof, is the following notion: A Noetherian local ring $(A,m)$ is called \textit{parafactorial} if $\dep(A) \geq 2$ and $\pic(\Spec(A) \setminus \{m\}) = 0$, where $\pic$ denotes the Picard group as defined in the appendix of Call's paper~\cite{Call}. This group is a common generalization of several important notions of class group, for instance:

\begin{enumerate}
\item $\pic(\Spec(A))$ is isomorphic to the usual Picard group of $A$.
\item If $A$ is in addition a normal domain and $G$ is the set of all height-one prime ideals of $A$ then $\pic(G)$ is isomorphic to the divisor class group of $A$.
\end{enumerate}

A central result in the works of Grothendieck~\cite[Théorème XI.3.13(ii)]{SGA2} and Call~\cite[Theorem 5]{Call} is the following fact which plays a critical role also in the present paper. It is said that it was in order to prove this that Grothendieck invented local cohomology.

\begin{fact}\label{remark:parafactorial}
A complete intersection of dimension at least $4$ is parafactorial.
\end{fact}

Our basic goal is to suggest a new perspective on the Samuel--Grothendieck Theorem. Since this result deals with Noetherian rings, we will restrict ourselves to this case. Hence, in what follows, 
\[
A \text{ will always denote a Noetherian commutative ring,} 
\]
possibly with some additional properties like being local, complete intersection, or Cohen--Macaulay, respectively. We recall from Bourbaki~\cite{Bourbaki} that, in the Noetherian case, a \textit{unique factorization domain} (UFD) can be defined as a normal domain with trivial divisor class group. It can be seen with relative ease that this is indeed equivalent to the property of a UFD of (essentially) unique factorizations into prime elements.

Following this philosophy, we are then able to view the Samuel--Grothendieck Theorem as the union of two independent parts corresponding to the two independent properties of a UFD above. The first part, which is represented by Corollary~\ref{theorem:Krull}, says that a local Cohen--Macaulay ring (so, in particular, a local complete intersection) is a normal domain if and only if it is normal in codimension at most $1$. The second part (Theorem~\ref{theorem:classgroup}) shows that the divisor class group of a normal local complete intersection domain can be described using information in codimension at most $3$.
Combining these two observations, we are able to reobtain the Samuel--Grothendieck Theorem as a corollary. 

In a final section, we investigate, for an integral Noetherian normal scheme $X$, the loss of information when passing from Weil to Cartier divisors. It is a classical theorem that these two groups of divisors (and their respective class groups) are isomorphic when $X$ is locally factorial, that is, all its local rings are UFDs~\cite[Ch. II, Prop. 6.11]{Hartshorne}. We show that the group of Weil divisors modulo the group of Cartier divisors of $X$ is a subgroup of the inverse limit of the divisor class groups of the local rings of $X$. Using Theorem~\ref{theorem:classgroup} in the case where $X$ is locally a complete intersection, we can restrict this limit to local rings of Krull dimension at most $3$.

\section{Noetherian local rings that are normal domains}

In this section, we give an easy characterization of Noetherian local rings that are normal domains by normality in codimension $\leq 1$. 

\begin{remark}\label{remark:domain} \cite[Exercise 9.11]{Matsumura}
Let $A$ be a Noetherian ring and $p_1,\ldots,p_r$ the minimal prime ideals of $A$. If the localization $A_{p}$ is an integral domain for each prime ideal $p$ of $A$ then $A = A/p_1 \times \ldots \times A/p_r$. In particular, if $A$ is in addition local then $A$ is itself an integral domain.
\end{remark}

Replacing ``normal domain'' by ``UFD'', the following is analogous to a result by Grothendieck~\cite[XI Corollaire 3.10]{SGA2} that can also be found in the paper of Call~\cite[Proposition 2]{Call}.

\begin{proposition}\label{proposition:Krull}
A Noetherian local ring $(A,m)$ of dimension at least $2$ is a normal domain if and only if $\dep(A) \geq 2$ and $A_p$ is a normal domain for all prime ideals $p \neq m$ of $A$.
\end{proposition}

\begin{proof}
First, let $A$ be a normal domain. It is clear that every localization of a normal domain is again normal. Moreover, $\dep(A) \geq 2$ by $\dim(A) \geq 2$ and Serre's criterion for normality, see for instance the textbook by Matsumura~\cite[Theorem 23.8]{Matsumura}.

For the converse, suppose that $\dep(A) \geq 2$ and $A_p$ is a normal domain for all $p \neq m$. Since $A_p$ is a normal domain for all $p \neq m$, Serre's criterion for normality is satisfied in these cases. The assumptions $\dim(A) \geq 2$ and $\dep(A) \geq 2$ cover the case $p = m$ and whence $A$ is normal. It follows by Remark~\ref{remark:domain} that $A$ is a domain and, therefore, $A$ is a normal domain.
\end{proof}

The following has already been noted by Geroldinger, Kainrath and Reinhart~\cite[Example 5.7.1]{GKR} in the case of integral domains. Moreover, it can also be easily deduced from Serre's criterion for normality directly. We prove it for Noetherian local rings in the spirit of Grothendieck~\cite[Corollaire XI 3.14]{SGA2} and Call~\cite[Theorem 7]{Call} using our Proposition~\ref{proposition:Krull}. 

\begin{corollary}\label{theorem:Krull}
Let $A$ be a local Cohen--Macaulay ring. $A$ is a normal domain if and only if $A_p$ is a normal domain for all prime ideals $p$ of $A$ of height at most $1$.
\end{corollary}

\begin{proof}
We show that $A_p$ is a normal domain by induction on $\hei(p)$, $p \in \Spec(A)$. This suffices because $A$ is local. If $\hei(p) \in \{0,1\}$, this is true by hypothesis. Suppose that $\hei(p) \geq 2$ and that $A_q$ is a normal domain for all prime ideals $q$ of $A$ that are properly contained in $p$. As a localization of a Cohen--Macaulay ring, $A_p$ is itself Cohen--Macaulay. So $\dep(A_p) = \dim(A_p) \geq 2$. Hence, $A_p$ is a normal domain by the induction hypothesis and Proposition~\ref{proposition:Krull}.
\end{proof}

\section{The divisor class group of complete intersection normal domains}\label{section:classgroup}

For a non-negative integer $i$, denote by $\Spec^i(A)$ the set of prime ideals of height $i$ of the ring~$A$. A basic property of a normal domain $A$ is $A = \bigcap_{p \in \Spec^1(A)} A_p$. A \textit{subintersection} of $A$ is a ring of the form $A' = \bigcap_{p \in \mathcal{P}} A_p$, where $\mathcal{P} \subseteq \Spec^1(A)$. Note that every localization of a normal domain is a subintersection. 

By $\Cl(A)$ we denote the divisor class group of $A$. For details on divisorial ideals and divisor class groups, we refer to Bourbaki's Commutative Algebra~\cite{Bourbaki}. The following is a theorem by Nagata. 

\begin{fact}\label{remark:Nagata} (see \cite[Theorem 7.1]{Fossum})
Let $A$ be a normal domain and $A'$ be a subintersection of $A$. Then the map
\begin{align*}
\Cl(A) &\to \Cl(A')\\
[I] &\mapsto [IA']
\end{align*}
is a surjective homomorphism of divisor class groups.
\end{fact}

\begin{remark}\label{remark:limit}
If $A$ is a normal domain and $p \subseteq q$ are prime ideals of $A$ then we denote the map from Fact~\ref{remark:Nagata} by $\phi_{p,q}: \Cl(A_q) \to \Cl(A_p)$. Viewing the projective limit
\[ \varprojlim_{\substack{ {p \in \Spec^i(A)} \\ {i \leq 3} }} \Cl(A_p) \]
with respect to the projective system given by the $\phi_{p,q}$, where $p\subseteq q$ are of height at most $3$, as a subgroup of the product of the $\Cl(A_p)$, we are given the map
\[
\phi: \Cl(A) \to \varprojlim_{\substack{ {p \in \Spec^i(A)} \\ {i \leq 3} }} \Cl(A_p)
\]
that is defined component-wise by $\phi_{p,m}$, in case $(A,m)$ is local. 
\end{remark}

As a main result, we show in our situation of interest that the map $\phi$ of Remark~\ref{remark:limit} is injective and that $\Cl(A)$ can therefore be seen as a subgroup of $\varprojlim_{\hei(p) \leq 3} \Cl(A_p)$. 
In order to apply Fact~\ref{remark:parafactorial}, we first need a lemma. For the notions of $G$-invertibility and $\pic(G)$ used in the proof, see the appendix of Call's paper~\cite{Call}.

\begin{lemma}\label{lemma:parafactorial}
Let $(A,m)$ be a Noetherian local normal domain of dimension at least $2$ and suppose that $A$ is parafactorial. Let $I$ be an ideal of $A$.

If $I_p$ is a principal ideal of $A_p$ for all prime ideals $p \neq m$ of $A$ then $I$ is a principal ideal of $A$. 
\end{lemma}

\begin{proof}
Set $G = \Spec(A) \setminus \{m\}$. As $A$ is Noetherian, $I$ is finitely generated. The assumption on the $I_p$ therefore implies that $I$ is $G$-invertible. So, it makes sense to consider $[I]_G$ which is the zero class because $\pic(G) = 0$ by definition of the term parafactorial. Applying the canonical map $\pic(G) \to \pic(\Spec^1(A))$ gives $[I]_{\Spec^1(A)} = 0$. Since $\pic(\Spec^1(A)) = \Cl(A)$, this means that $I$ is principal.
\end{proof}

\begin{theorem}\label{theorem:classgroup}
Let $A$ be a normal local domain that is a complete intersection.
 
Then the divisor class group of $A$ is a subgroup of \[\varprojlim_{\hei(p) \leq 3} \Cl(A_p), \] the projective limit of the divisor class groups in codimension at most $3$ with respect to the projective system of Remark~\ref{remark:limit}.
\end{theorem}

\begin{proof}
If $\dim(A) \leq 3$, this is trivial. So suppose that $\dim(A) \geq 4$. We only need to show the following in order to prove that the map $\phi: \Cl(A) \to \varprojlim_{\hei(p) \leq 3} \Cl(A_p)$ of Remark~\ref{remark:limit} is injective:

\begin{itemize}
\item[(1)] Every divisorial ideal $I$ of $A$ with the property that $I_p$ is a principal ideal of $A_p$, for all prime ideals $p$ with $\hei(p) \leq 3$, is a principal ideal of $A$.
\end{itemize}

$A$ is parafactorial by Fact~\ref{remark:parafactorial}. Let $I$ be a divisorial ideal of $A$ and suppose that

\begin{itemize}
\item[(2)] \label{(2)} $I_p$ is a principal ideal of $A_p$ for all prime ideals $p$ of $A$ of height at most $3$.
\end{itemize}

In view of Lemma~\ref{lemma:parafactorial}, it suffices to show that $I_p$ is principal in case $\hei(p) \geq 4$, $p \neq m$. We do this by induction on $\hei(p)\geq 4$, $p \neq m$. As a localization of a complete intersection, $A_p$ is a complete intersection, see~\cite{complete-int-loc}. So $A_p$ is parafactorial by Fact~\ref{remark:parafactorial}. By the induction hypothesis together with~(2), we see that Lemma~\ref{lemma:parafactorial} applies to the ideal $I_p$ of $A_p$. Hence $I_p$ is principal.
\end{proof}

Corollary~\ref{theorem:Krull} and Theorem~\ref{theorem:classgroup} immediately imply the Samuel--Grothendieck Theorem.

\begin{corollary}
Let $A$ be a local complete intersection ring that is a UFD in codimension at most $3$. Then $A$ is a UFD.
\end{corollary}

In view of arithmetic applications of Theorem~\ref{theorem:classgroup} it is crucial to understand the distribution of prime divisors over the divisor classes. We therefore want to suggest the following open problem.

\begin{problem}
Describe for a local complete intersection ring $A$ the set of those classes in $\Cl(A) \subseteq \varprojlim_{\hei(p) \leq 3} \Cl(A_p)$ that contain height-one prime ideals of $A$. 
\end{problem}

\section{The gap between Weil and Cartier divisors}

Let $X$ be an integral Noetherian normal scheme. Then it makes sense to form both, the group $\dw(X)$ of Weil divisors and the group $\dc(X)$ of Cartier divisors on $X$. Moreover, the Cartier divisors can be identified with those Weil divisors that are locally principal~\cite[Ch. II, Rem. 6.11.2]{Hartshorne}. Under this identification, the principal divisors coincide and hence it induces an embedding of class groups $\Pic(X) = \Clc(X) \subseteq \Clw(X)$.

We want to note that, in his introductory textbook on algebraic geometry~\cite{Hartshorne}, Hartshorne introduces divisors only in the case when $X$ is also separated. However, in later works, he developed the theory more generally and managed to avoid this assumption, see for instance his 1994 paper about generalized divisors on Gorenstein schemes~\cite{Hartshorne94}.

If $Y$ is an integral Noetherian normal scheme that is in addition affine, $\dw(Y)$ is the usual group of divisors of the normal domain $\mathcal{O}(Y)$ of global sections on $Y$ and \[\Clw(Y) \cong \Cl(\mathcal{O}(Y)).\]
Using this, for every point $x \in X$, we have a surjective homomorphism from $\Clw(X)$ onto the divisor class group $\Cl(\mathcal{O}_{X,x})$ of the local ring at $x$ and these maps are compatible with those between these divisor class groups as described in Fact~\ref{remark:Nagata} and Remark~\ref{remark:limit}.
Now, the loss when passing from Weil to Cartier divisors can be described in the following way. This seems to be well-known, nevertheless, we spell out a short proof for sake of completeness.

\begin{proposition}\label{proposition:Weil/Cartier}
Let $X$ be an integral Noetherian normal scheme. Then the factor group $\dw(X)/ \dc(X)$ is a subgroup of
\[ \varprojlim_{x \in X} \Cl(\mathcal{O}_{X,x}), \]
the projective limit of the divisor class groups of the local rings of $X$.
\end{proposition}

\begin{proof}
The maps $\Clw(X) \to \Cl(\mathcal{O}_{X,x})$ induce a map \[\Clw(X) \to \varprojlim_{x \in X} \Cl(\mathcal{O}_{X,x}).\] Since the Cartier divisors on $X$ are exactly the locally invertible Weil divisors, the kernel of this map coincides with $\Clc(X) = \Pic(X)$ and hence
\[
\dw(X) / \dc(X) \cong \Clw(X) / \Pic(X) \subseteq \varprojlim_{x \in X} \Cl(\mathcal{O}_{X,x}).\]
\end{proof}

It is worth mentioning that, looking back to our introductory remarks on non-unique factorizations, the factor group of Weil modulo Cartier divisors of Proposition~\ref{proposition:Weil/Cartier} also appears in this context very naturally: A recent result of Geroldinger and Khadam~\cite[Theorem 4.3]{idealmonoid} shows, in particular, that it is the divisor class group of the monoid of invertible ideals of the ring of global sections of $X$, which gives a connection of our observations with the study of non-unique factorizations in this monoid.

We recover from Proposition~\ref{proposition:Weil/Cartier} the following classical result.

\begin{corollary}~\cite[Ch. II, Prop. 6.11]{Hartshorne}
Let $X$ be an integral Noetherian normal scheme, all of whose local rings are unique factorization domains. Then the groups of Weil divisors and of Cartier divisors on $X$ coincide.
\end{corollary}

We use our description of the class group of a local complete intersection ring (Theorem \ref{theorem:classgroup}) to immediately obtain the following corollary. Here we call $X$ locally a complete intersection if all local rings are complete intersections. 

\begin{corollary}
Let $X$ be an integral Noetherian normal scheme and suppose that $X$ is locally a complete intersection. Then the factor group $\dw(X)/ \dc(X)$ is a subgroup of
\[ \varprojlim_{\substack{x \in X \\ \dim(\mathcal{O}_{X,x}) \leq 3}} \Cl(\mathcal{O}_{X,x}), \]
the projective limit of the divisor class groups of the local rings of Krull dimension at most $3$ of~$X$.
\end{corollary}

\section*{Acknowledgements} 
I want to thank Victor Fadinger-Held and Alfred Geroldinger for interesting discussions on the topic. I am also very grateful to Salvatore Tringali and to the anonymous referee for their helpful comments on the presentation of the manuscript.

\bibliographystyle{amsplainurl}
\bibliography{bibliography}
 
\vspace{1cm} 
 
\noindent
\textsc{Daniel Windisch\\
School of Mathematics, University of Edinburgh, James Clerk Maxwell Building, Peter Guthrie Tait Road, EH9 3FD,  Edinburgh, United Kingdom} \\
\textit{E-mail address}: \texttt{Daniel.Windisch@ed.ac.uk}

\end{document}